\documentclass{amsart}
\usepackage{amssymb}
\newtheorem{lem}{Lemma}
\newtheorem{Theo}{Theorem}
\newtheorem{prop}{Proposition}
\newcommand{\Q}{\mathbb{Q}}
\newcommand{\Z}{\mathbb{Z}}
\newcommand{\N}{\mathbb{N}}
\newcommand{\R}{\mathbb{R}}

\begin{document}
\baselineskip=17pt
\title{The irrationality of some number theoretical series}
\author{J.-C. Schlage-Puchta}
\maketitle

MSC-Index: 11J72\\[3mm]

In this note we prove the irrationality of some series by
combining methods from elementary and analytic number theory with
methods from the theory of uniform distribution.
Our first results yields an explicit set of uncountably many
$\Q$-linearly independent real numbers.
\begin{Theo}
For a real number $\lambda\geq0$ define the series
\[
S_\lambda=\sum_{n\geq0}\frac{[n^\lambda]}{n!}.
\]
Then the set $\{1, e\}\cup\{S_\lambda:\lambda\in(0, \infty)\setminus\Z\}$ is $\Q$-linearly independent.
\end{Theo}
Some properties of the function $S_\lambda$ are given in the following.
\begin{prop}
The function $\lambda\mapsto S_\lambda$ is injective, monotone, and continuous from the
right. Its image has cardinality of the continuum, Hausdorff dimension
0, and it is totally disconnected.
\end{prop}
Our second Theorem deals with real numbers defined by
their digital expansion.
\begin{Theo}
Let $b\geq 2$ be an integer, not a proper power.
Let $g:\N\to\R$ be a continuous non-decreasing function, such that
$\frac{g(n+1)}{g(n)}\to 1$. Let $f:\N\to\N$ be a non-decreasing function,
such that $\frac{f(n+1)}{f(n)}\sim g(n)$, and denote by $\alpha$ the real
number obtained by writing out its digits to base $b$ as
$0.f(1)f(2)f(3)\ldots$. Suppose that $\alpha$ is rational. Then $g$
converges to some constant $c$, $c$ is a power of $b$, and
$f(n+1)=cf(n)+\mathcal{O}(1)$.
\end{Theo}
Note that there do exist functions $f, g$ such that $\alpha$ is rational,
for example, taking $b=10$, $f(n)=\frac{10^n-1}{9}$, and $g(n)\to10$ we find
$\alpha=\frac{1}{9}$. For the function $f(n)=a^n$, $a\in\N$, $a\geq 2$, this
result was proven for base $b=10$ by Mahler \cite{Mahler} and for
arbitrary $b$, including the case of $b$ being a proper power,
by Bundschuh \cite{Bund}. 

Denote by $p_n$ the $n$-th prime number, and, for an integer $k\geq 0$,
define the series $S_k$ by
\[
S_k := \sum_{n=1}^\infty \frac{p_n^k}{n!}.
\]
P. Erd{\"o}s\cite{Erd} stated that $S_k$ is irrational and gave a proof
for $k=1$. However, it appears that, for $k>1$, no proof has appeared
in print. 
Our last result is the following.
\begin{Theo}
The real numbers $1, S_0, S_1, S_2, \ldots$ are $\Q$-linearly independent.
\end{Theo}
Our proofs will use the following results from the theory of
equidistribution, confer, for example, \cite[Theorem~2.8]{GK} and
\cite[II, Theorem~2.5]{KN}. 
\begin{lem}[Weyl-Van der Corput]
\label{Lem:WeylCorput}
Let $f$ be a function, which is $(q+2)$-times continuously differentiable,
and suppose that on the interval $[1, N]$ we have
$\lambda\leq f^{(q+2)}(t)\leq\alpha\lambda$. Then for $q=0$ we have
\[
\sum_{n=1}^{N} e(f(n)) \ll \alpha N\lambda^{1/2} + \lambda^{-1/2},
\]
whereas for $q\geq 1$ we have
\[
\sum_{n=1}^{N} e(f(n)) \ll N(\alpha^2\lambda)^{\frac{1}{4Q-2}} +
  N^{1-\frac{1}{2Q}}\alpha^{\frac{1}{2Q}} + N^{1-\frac{1}{2Q}+\frac{1}{Q^2}}\lambda^{-\frac{1}{2Q}},
\]
where $Q=2^q$.
\end{lem}
\begin{lem}[Erd\H os-Tur\'an]
\label{Lem:ErdTuran}
Let $(x_n)_{n=1}^N$ be a sequence of real numbers in the interval $[0,
1]$. Then the discrepancy $D_N$ of this sequence is bounded above by
\[
D_N \ll \frac{N}{H} + \sum_{1\leq h\leq H}\left|\sum_{n=1}^N e(hx_n)\right|
\]
\end{lem}
\begin{proof}[Proof of Theorem 1.]
It suffices to show that for all tuples $a_1, \ldots, a_k\in\Z\setminus\{0\}$,
$0\leq\lambda_1<\lambda_2<\dots<\lambda_k$, such that no $\lambda_i$ is an integer $\geq 2$, the real number
\[
S = \sum_{n\geq0} \frac{1}{n!}\big(a_1[n^{\lambda_1}] + \dots + a_k[n^{\lambda_k}]\big)
\]
is irrational. Moreover, we may assume that at least one of the $\lambda_i$
is not 0. Suppose that $S=\frac{p}{q}$, and let $n\geq q$ be an
integer. Then $n!\cdot S$ is integral, and we deduce
\[
\left\|\sum_{\nu\geq1}\frac{1}{(n+1)\dots(n+\nu)} \big(a_1\big[(n+\nu)^{\lambda_1}\big] + \dots
  + a_k\big[(n+\nu)^{\lambda_k}\big]\big)\right\| = 0,
\]
where $\|\cdot\|$ denotes the difference to the nearest integer.
Set $M=[\lambda_k]+1$. Then truncating the series at $\nu=M$ yields an error
of size $\mathcal{O}(n^{-1})$. Neglecting the rounding introduces an
error of the same magnitude, and we obtain
\begin{equation}
\label{eq:disclower}
\left\|\sum_{\nu=1}^M\frac{1}{(n+1)\dots(n+\nu)} \big(a_1(n+\nu)^{\lambda_1} + \dots
  + a_k(n+\nu)^{\lambda_k}\big)\right\| \ll \frac{1}{n}
\end{equation}
If $\lambda_k<1$, the sum collapses to a single term, which tends to 0, and
we obtain $n^{\lambda_k-1}\ll n^{-1}$, contradicting the assumption that at
least one of the $\lambda_i$ is nonzero. If $\lambda_k>1$, define
\[
f(t) = \sum_{\nu=1}^M\frac{1}{(t+1)\dots(t+\nu)} \big(a_1(t+\nu)^{\lambda_1} + \dots
  + a_k(t+\nu)^{\lambda_k}\big).
\]
For $t>M$, $f$ is analytic, and, since $\lambda_k$ is not integral, there
exist some $K\in\N$ such that $f^{(K+1)}(t)$ and $f^{(K+2)}(t)$ do not
change sign for $t>t_0$, and
\[
\lim_{t\to\infty} \frac{f^{(K)}(t)}{t} = 0,\qquad \lim_{t\to\infty} \frac{1}{tf^{(K+1)}(t)} = 0.
\]
Lemma~\ref{Lem:WeylCorput} now implies that the sequence $f(n)$ is
equidistributed modulo 1; confer
e.g. \cite[pp. 36--39]{Hlawka}. However, the latter statement clearly
contradicts (\ref{eq:disclower}), which proves our theorem.
\end{proof}
\begin{proof}[Proof of Proposition 1.]
Monotonicity is clear. Suppose that $\lambda_2>\lambda_1$. Then for $n\geq n_0$ we
have $n^{\lambda_2}>n^{\lambda_1}+1$, and therefore $[n^{\lambda_2}]\geq[n^{\lambda_1}]+1$, which
implies $S_{\lambda_2}\geq S_{\lambda_1} + \frac{1}{n_0!}$, and we conclude that the
map $\lambda\mapsto S_\lambda$ is injective. Hence, the imagy of this map has the same
cardinality as its range, which is the continuum. We now prove continuity
from the right. For each $n$ there is some $\epsilon_n$ such that $[n^\lambda]=[n^t]$ for all
$t\in[\lambda, \lambda+\epsilon]$. Let $N$ be a sufficiently large integer, and set
$\epsilon=\min(1, \{\epsilon_n: n\leq N\})$. Then we have for $t\in[\lambda, \lambda+\epsilon]$ the bound
\[
S_t-S_\lambda = \sum_{n=1}^\infty \frac{[n^t]-[n^\lambda]}{n!} = \sum_{n=N+1}^\infty
\frac{[n^t]-[n^\lambda]}{n!} \leq \sum_{n=N+1}^\infty \frac{n^{\lambda+1}}{n!} \leq 2\frac{N^\lambda}{N!},
\]
which tends to 0 as $N\to\infty$. Hence, if $\lambda_i\searrow\lambda$, then $S_{\lambda_i}\searrow S_\lambda$, and
$S_\lambda$ is continuous from the right. The fact that the image of
$S_\lambda$ is totally disconnected follows from Theorem~1, since a
connected component would contain some interval of positive length,
and therefore infinitely many rational numbers, only finitely many of
which are excluded by the condition $\lambda\not\in\Z$. 

To estimate the Hausdorff dimension, let $N$ be an integer, and define
$S_\lambda^N$ as the partial sum $\sum_{n=1}^N \frac{[n^\lambda]}{n!}$. The image
$\mathcal{I}_N$ of an interval $[t, t+1]$ under the map $\lambda\mapsto S_\lambda^N$
consists of finitely many points, more precisely, the cardinality of
the image of $[t, t+1]$ is at most equal to the number of values
$\lambda\in[t, t+1]$, such that $n^\lambda$ is integral for some $n\leq N$, and this
quantity is bounded above by $N^{t+2}$. Clearly, $S_\lambda^N\leq S_\lambda$, and,
for $\lambda\in[t, t+1]$ we have for every fixed $A>0$ and $N$ sufficiently
large the estimate
\[
|S_\lambda^N-S_\lambda| = \sum_{n>N}\frac{[n^\lambda]}{n!} < \sum_{n>N}\frac{n^{t+1}}{n!} <
\frac{N^{t+1}}{N!} < N^{-A},
\]
Hence, for $N$ large the image $\mathcal{I}$ of $[t, t+1]$ under the
map $\lambda\mapsto S_\lambda$ can be convered by $N^{t+2}$ intervals of length $N^{-A}$
each, thus, the Hausdorff dimension of $\mathcal{I}$ is bounded above
by $\frac{t+2}{A}$ for any $A$, and therefore the Hausdorff dimension
is 0.
\end{proof}
\begin{proof}[Proof of Theorem 2.]
For a positive integer $n$, denote by $\ell(n)$ the number of digits of $n$,
and by $\log n$ the logarithm in base $b$. Suppose that $\alpha$ is
rational. Then the sequence of digits of $\alpha$ is ultimately
periodic with period $p$, say.

There are only $p$ cyclic permutations of the digits of the period of
$\alpha$, hence, $\log f(n) \bmod 1$ has at most $p$ limit points; order
these limit points as
$0\leq x_1<x_2<\dots<x_m<1$. In particular, for every $\epsilon>0$ there exists
some $n_0$ such that for $n>n_0$ we have $f(n)-x_i\bmod 1 <\epsilon$ for some
$i$ depending on $n$. Moreover, increasing $n_0$, if necessary, we may
assume that $|\log f(n+1)-\log f(n)-\log g(n)|<\epsilon$, and obtain
\[
\log g(n)\bmod 1 = \log f(n+1)-\log f(n)\bmod 1 +\delta_1  = x_i - x_j +
\delta_2
\]
for some indices $i, j$, and real numbers $0\leq\delta_1, \delta_2<\epsilon$.
Hence, $\log g(n)\bmod 1$ has finitely many limit points as
well. However, since $\log g(n+1)-\log g(n)\to 0$, this
implies that $\log g(n)$ converges, thus, $g(n)\to c$ for some constant $c$.

We now distinguish two cases, depending on whether $\frac{\log c}{\log
  b}$ is rational or not. 

Suppose that $\frac{\log c}{\log b}$ is rational. Since
$\frac{f(n+1)}{f(n)}$ converges, all but finitely many $n$ have the
property that there are infinitely many $m$ such that
$\ell(f(n))\equiv\ell(f(m))\pmod{p}$, and that $f(n)$ and $f(m)$ begin
with the same $p$ digits. This implies in particular that $f(n+1)$ and
$f(m+1)$ begin with the same $p$ digits. Since
$\frac{f(m+1)}{f(m)}\to c$, and there are infinitely many $m$
at our disposal, we may choose $m$ so large that 
\[
\left|\frac{f(m+1)}{f(m)}-c\right|<\frac{1}{f(n)},
\]
while periodicity implies
\[
\left|\frac{f(m+1)}{f(m)}-\frac{f(n+1)}{f(n)}\right|<\frac{1}{f(n)},
\]
hence, $f(n+1)=cf(n)+\mathcal{O}(1)$ holds true for all $n$.
Furthermore, for a sequence $n_i$ such that $\ell(f(n_i))\pmod{p}$ and
the first $p$ digits of $f(n_i)$ are constant, the rational numbers
$\beta$ and $\gamma$ obtained from $\alpha$ by shifting the decimal
point right in front of the first digit belonging to $f(n_i)$
resp. $f(n_i+1)$ do not depend on $i$. Hence, 
\[
c=\lim_{i\to\infty}\frac{f(n_i+1)}{f(n_i)} =
\lim_{i\to\infty}\frac{\beta b^{\ell(f(n_i+1))}+\mathcal{O}(1)}
{\gamma b^{\ell(f(n_i))} + \mathcal{O}(1)} = \frac{\beta b^{\ell(f(n_i)+1) - \ell(f(n_i)+1)}}{\gamma}
\]
is rational. Hence, $c$ is both rational and a rational power of $b$,
which either implies that $b$ is  a proper power, or that $c$ is a
proper power of $b$. Hence, our theorem holds true in this case.

Now we suppose that $\frac{\log c}{\log b}$ is irrational; we have to
show that this assumption leads to a contradiction.
For any irrational number $\alpha$ the sequence $(\frac{\alpha n}{p})_{n=1}^\infty$ is
equidistributed modulo 1, in particular, there are
infinitely many $n$ such that $\ell(f(n))$ is divisible by $p$ and
$f(n)$ does not have both its leading digits equal to $b-1$. For such
$n$, $f(n)$ 
and $f(n+1)$ begin with the same digits, in fact, $f(n)$ is an initial
segment of $f(n+1)$. Hence, $f(n+1)=f(n)\cdot b^k + \mathcal{O}(b^k)$
for some positive integer $k$. Since $g(n)\to c$, we deduce that $c$
itself is a power of $b$, contradicting the assumption that
$\frac{\log c}{\log b}$ is irrational, and the theorem is proven.
\end{proof}
For the proof of Theorem~3 we need some auxilliary
results. The first result is consequence of Selberg's sieve, confer,
e.g. \cite[Theorem~5.1]{HR}
\begin{lem}
\label{Lem:Selberg}
Let $0\leq a_1 < a_2 < \dots < a_k<N$ be a sequence of integers, and let
$\mathcal{N}\subseteq[x, 2x]$ be a set of integers such that for all
$n\in\mathcal{N}$ each of the integers $n+a_i$ is prime. Then
\[
|\mathcal{N}| \leq \frac{c_kx}{\log^{k+1}x} \prod_p
\left(1+\frac{1}{p}\right)^{k+2-\nu(p)},
\]
where $\nu(p)$ denotes the number of distinct residues $\bmod\;p$ among
\[
\{0, \Delta_0, \Delta_0+\Delta_1, \ldots, \Delta_0+\dots+\Delta_k\},
\]
and $\log_2$ denotes the iterated logarithm. In particular,
$|\mathcal{N}|\ll_k\frac{x\log_2^{k+2}}{\log^{k+1}x}$. 
\end{lem}
For the rest of this article, set $\delta_n=p_{n+1}-p_n$. 
\begin{lem}
\label{lem:sieve}
Let $F\in\Z[x_0, \ldots, x_k]$ be a polynomial which does not vanish
identically. Then for almost all $n$ we have $F(\delta_n, \ldots,\delta_{n+k})\neq 0$.
\end{lem}
\begin{proof}
Neglecting
$\mathcal{O}\big(\frac{x}{\log_2 x}\big)$ indices at most, we
may assume that $\delta_i\leq\log x\log\log x$, $n\leq i\leq n+k$. For a
fixed tuple $\Delta_0, \ldots, \Delta_k$ satisfying
$\Delta_i\leq\log x\log_2 x$, the number of solutions $n$ of the
equations $\delta_{n+i}=\Delta_i, 0\leq i\leq k$, is bounded above by
the number $N$ of primes $p\in[p_x, p_{2x}]$ with the property that
$p+\Delta_0+\dots+\Delta_i$ is prime for all $0\leq i\leq k$, and from
Lemma~\ref{Lem:Selberg} we infer that this quantity is
$\mathcal{O}(\frac{x\log_2^{k+2} x}{\log^{k+1} x})$. Since $F$ does
not vanish identically, the number of tuples $(\Delta_0, \ldots,\Delta_k)$ such that
$F(\Delta_0, \ldots, \Delta_k)=0$ and $\Delta_i\leq\log x\log_2 x$ for all $i$, is of size
$\mathcal{O}(\log^k x\log_2^k x)$, hence, the total number of
solutions of the equation $F(\delta_n, \ldots, \delta_{n+k})=0$ is
$\ll\frac{x\log_2^{2k+2} x}{\log x}$, which is sufficiently small. 
\end{proof}
\begin{lem}
\label{lem:Pnonvan}
Let $k$ be a field, $P, Q\in k[X_1, \ldots, X_n]$ be polynomials, $\nu\neq 0$ an
integer, such that
\begin{multline*}
\nu X_1P(X_1, \ldots, X_n) + P(X_1, \ldots, X_n) Q(X_2, \ldots, X_{n+1})
\\
- P(X_2, \ldots,X_{n+1}) Q(X_1, \ldots, X_n)
\end{multline*}
vanishes identically. Then $P$ vanishes identically.
\end{lem}
\begin{proof}
Suppose $Q\neq 0$, and put $X_1=0$. Then the polynomial
\[
P(0,X_2,  \ldots, X_n) Q(X_2, \ldots, X_{n+1}) -
P(X_2, \ldots,X_{n+1}) Q(0, X_2, \ldots, X_n)
\]
vanishes identically, that is, putting $R(X_1, \ldots, X_n)=\frac{P(X_1, \ldots,
  X_n)}{Q(X_1, \ldots, X_n)}$ we find that
\[
R(0, X_2, \ldots, X_n) = R(X_2, \ldots, X_{n+1})
\]
holds identically. In particular, $R(X_1, \ldots, X_n)$ does not involve
$X_n$ at all. Hence,
\[
R(0, X_2, \ldots, X_{n-1}, 0) = R(X_2, \ldots, X_n, 0)
\]
holds identically, and we deduce that $R$ does not involve $X_{n-1}$
either. Continuing in this way we obtain that $R$ is constant, that
is,
\[
P(X_1, \ldots, X_n) Q(X_2, \ldots, X_{n+1}) -
P(X_2, \ldots,X_{n+1}) Q(X_1, \ldots, X_n)
\]
and therefore $\nu X_1P(X_1,\ldots, X_{n})$ vanishes identically. However,
since $\nu\neq 0$, this can only happen if $P$ vanishes.
\end{proof}
\begin{lem}
\label{lem:disclow}
Let $Q\in\Z[X]$ be a non-constant polynomial of degree $d$ and
coefficients bounded by $M$. Then the discrepancy $D$ of the sequence
$\big(Q\big(\frac{p_n}{n}\big)\bmod 1\big)_{n=x}^{2x}$ is bounded above by
\[
D\ll xe^{-c\sqrt{\log x}} + M^{1/3} x^{2/3}\log^{d/3} x
\]
\end{lem}
\begin{proof}
For $t\geq 2$, denote by $f(t)$ the inverse function of $\mathrm{li}\,t$, that is,
the unique positive solution of the equation
\[
\int_2^{f(t)} \frac{dx}{\log x} = t.
\]
Then, by the prime number theorem, we deduce that 
\[
Q\Big(\frac{p_n}{n}\Big)-Q\Big(\frac{f(n)}{n}\Big) \ll e^{-c\sqrt{\log n}},
\]
hence, it suffices to estimate the discrepancy $\tilde{D}$ of the sequence
\[
\big(Q\big(f(n)n^{-1})\bmod 1\big)_{n=x}^{2x}.
\]
We will do so using Lemma~\ref{Lem:ErdTuran}. Set
$F(t)=Q(f(t)t^{-1})$. Then we have for every integer $H\geq 1$
\[
\tilde{D}\ll\frac{x}{H} + \frac{1}{h}\sum_{h\leq H}\left|\sum_{n=x}^{2x} e(hF(n))\right|.
\]
The second derivative of $hF$ is of size $ht^{-1}M\log^d t$.
We can now apply the case $q=0$ of Lemma~\ref{Lem:WeylCorput} to deduce
\[
\sum_{n=x}^{2x} e(hF(n)) \ll M^{1/2}x^{1/2}h^{1/2}\log^{d/2} x,
\]
and therefore
\[
\tilde{D}\ll \frac{x}{H} + M^{1/2}x^{1/2}H^{1/2}\log^{d/2} x
\ll M^{1/3}x^{2/3}\log^{d/3} x.
\]
\end{proof}
\begin{proof}[Proof of Theorem 3.]
It suffices to show that
\[
S=\sum_{\nu=1}^\infty\frac{P(p_\nu)}{\nu!}
\]
is irrational for every polynomial $P$ with integral
coefficients which does not vanish identically. Assume to the contrary
that $S$ is rational. Then, for 
$n$ sufficiently large, $n!S$ is integral, and we deduce that
\[
\sum_{\nu=n+1}^\infty\frac{P(p_\nu)}{(n+1)\cdots\nu}
\]
is integral. Denote the degree of $P$ by $k$. Since $p_\nu\sim\nu\log\nu$, we
deduce
\[
\left\|\sum_{\nu=1}^{k-1}\frac{P(p_{n+\nu})}{(n+1)\cdots(n+\nu)}\right\|\ll\frac{\log^kn}{n}
\]
for all $n$ sufficiently large. Set
\[
F^{(0)}(n)=\sum_{\nu=1}^{k-1}\frac{P(p_{n+\nu})}{(n+1)\cdots(n+\nu)}.
\]
In the sequel we shall write $R(n)$ to denote a negligible error term,
that is, any function satisfying $R(n)\ll\frac{\log^c n}{n}$ for almost
all $n$ and some constant $c$. In particular, we have $\delta_n R(n) = R(n)$. 
Expanding $\frac{1}{(n+1)\cdots(n+k)}$ into a Laurent-series around 0, and
expressing $p_{n+k}$ by $p_n$ and the $\delta_{n+i}$'s, we find
\[
F^{(0)}(n) = \sum_{\nu=1}^k\sum_{\mu=1}^\nu P_{\nu\mu}^{(0)}(\delta_n, \ldots,
\delta_{n+k-1})\frac{p_n^\nu}{n^\mu} + R(n).
\]
Define a partial order on the set of pairs $(\nu, \mu)$ by $(\nu_1,
\mu_1)\succ(\nu_2, \mu_2)$, if $\nu_1-\mu_1>\nu_2-\mu_2$, or
$\nu_1-\mu_1=\nu_2-\mu_2$ and 
$\nu_1>\nu_2$; that is, the corresponding fraction
$\frac{p_n^{\nu_1}}{n^{\mu_1}}$ has faster growth than $\frac{p_n^{\nu_2}}{n^{\mu_2}}$.
We now define a sequence of functions $F_i$ and $P_{\mu\nu}^{(i)}$ recursively by 
\[
F^{(i+1)}(n) = P_{\nu_0\mu_0}^{(i)}(\delta_{n+1}, \ldots,
\delta_{n+k})F^{(i)}(n) - P_{\nu_0\mu_0}^{(i)}(\delta_n, \ldots, \delta_{n+k-1}) 
F^{(i)}(n+1)
\]
where $(\nu_0, \mu_0)$ is maximal with respect to $\succ$ among all pairs with the
property that $P_{\nu\mu}^{(i)}$ is non-trivial, and
\[
F^{(i)}(n) = \sum_{\nu=1}^{k-1}\sum_{\mu=1}^\nu P_{\nu\mu}^{(i)}(\delta_n, \ldots,
\delta_{n+k})\frac{p_n^\nu}{n^\mu} + R(n).
\]
We have
\begin{eqnarray*}
P_{\nu_0-1\mu_0}^{(i+1)} & = & \nu_0P_{\nu_0\mu_0}^{(i)}(\delta_n, \ldots, \delta_{n+k-1})\delta_n +
P_{\nu_0\mu_0}^{(i)}(\delta_{n+1}, \ldots,\delta_{n+k})P_{\nu_0-1\mu_0}^{(i)}(\delta_n, \ldots, \delta_{n+k-1})\\
&&\qquad - P_{\nu_0\mu_0}^{(i)}(\delta_n, \ldots, \delta_{n+k-1}) P_{\nu_0-1\mu_0}^{(i)}(\delta_{n+1}, \ldots, \delta_{n+k})
\end{eqnarray*}
Since $\nu_0\geq 1$, we deduce from Lemma~\ref{lem:Pnonvan} that
$P_{\nu_0-1\mu_0}^{(i+1)}$ does 
not vanish identically. In each step, the pair $(\nu_0, \mu_0)$ is removed
from the set of pairs occurring in $F^{(i)}$, and all pairs $(\nu, \mu)$
added satisfy $\nu-\mu<\nu_0-\mu_0$, thus, after finitely many steps the
maximum of $\nu-\mu$ is reduced by 1, and again after finitely many steps
the maximum is reduced to 0, that is, we reach some
$F^{(i)}$, such that $\mu=\nu$ for all $(\mu,\nu)$ with $P_{\mu\nu}\neq 0$, and
with at least one pair $(\mu, \nu)$ such that $P_{\mu\nu}\neq 0$. Moreover, we
have $\|F^{(i)}(n)\|= R(n)$. During the recursive process leading to
$F^{(i)}$, we multiplied the initial polynomial, which had rational
coefficients, with other polynomials with rational coefficients, 
and shifted indices. Hence, there exist some integer $\ell$ and
polynomials $Q_i\in\Q[X_1, \ldots,X_\ell]$, $1\leq i\leq \ell$, $Q_\ell\neq 0$, such that 
\[
\left\|\sum_{i=1}^\ell Q_i(\delta_n, \ldots, \delta_{n+\ell})\frac{p_n^i}{n^i}\right\|= R(n).
\]
Moreover, since $c\cdot R(n)=R(n)$ for every constant $c$, we may multiply
with all occurring denominators, and suppose that each $Q_i$ has
integral coefficients. Therefore, we can apply Lemma~ \ref{lem:sieve}
and find that for almost all $n$, one of the
polynomials $Q_i(\delta_n, \ldots, \delta_{n+\ell})$ does not vanish, while at
the same time, for almost all 
$n$ none of the differences $\delta_n, \ldots, \delta_{n+\ell}$ exceeds $\log^2
n$. Hence, setting $a_i=Q_i(\delta_n, \ldots, \delta_{n+\ell})$, we find that
for almost all $n$ there are integers $a_1, \ldots, a_\ell$, not all of which
vanish, with 
$0<|a_i|<\log^A n$ for some constant $A$ depending only on the initial
polynomial $P$, such that
\begin{equation}
\label{eq:aiungleichung}
\left\|\sum_{i=1}^\ell a_i\frac{p_n^i}{n^i}\right\|\ll e^{-c\sqrt{\log n}}.
\end{equation}
In particular, there are integers $a_1, \ldots, a_\ell$ such that
(\ref{eq:aiungleichung}) is satisfied for at least $\frac{x}{\log^{\ell A}
x}$ integers $n\leq x$ for some constant $A$. This clearly contradicts
Lemma~\ref{lem:disclow}, thereby proving Theorem 3.
\end{proof}

\noindent Jan-Christoph Schlage-Puchta\\
Mathematisches Institut\\
Eckerstr. 1\\
79104 Freiburg\\
Germany\\
jcsp@mathematik.uni-freiburg.de
\end{document}